\documentclass[12pt,intlimits]{article}
\usepackage[a4paper,top=2cm,bottom=2cm,left=2cm,right=2cm,bindingoffset=5mm]{geometry}
\usepackage{amsmath}
\usepackage{amsthm}
\usepackage{amssymb}

\usepackage{graphicx}
\usepackage{epsfig}
\usepackage{hyperref}
\hypersetup{colorlinks=true,linkcolor=blue}
\newtheorem{theo}{Theorem}[section]
\newtheorem{lem}[theo]{Lemma}
\newtheorem{defi}[theo]{Definition}

\newtheorem{prop}[theo]{Proposition}

\newcommand{\be}{\begin{eqnarray}}
\newcommand{\ee}{\end{eqnarray}}
\newcommand{\bes}{\begin{eqnarray*}}
\newcommand{\ees}{\end{eqnarray*}}
\newcommand{\bi}{\begin{itemize}}
\newcommand{\ei}{\end{itemize}}
\newcommand{\ben}{\begin{enumerate}}
\newcommand{\een}{\end{enumerate}}

\newcommand{\enne}{n\times n}

\newcommand{\R}{\mathbb{R}}
\newcommand{\N}{\mathbb{N}}

\newcommand{\de}{\mathrm {d}}

\def\einschr{\hbox{\kern1pt\vrule height 6pt\vrule  width6pt height 0.4pt depth0pt\kern1pt}}

\DeclareMathOperator{\dive}{div}

\DeclareMathOperator{\supp}{supp}
\DeclareMathOperator{\curl}{curl}

\DeclareMathOperator{\diam}{diam}

\title{\bf  Friedrichs inequality in irregular domains  }

\date{}

\begin{document}
\maketitle

\centerline{\scshape Simone Creo, Maria Rosaria Lancia}
\medskip
{\footnotesize

 \centerline{Dipartimento di Scienze di Base e Applicate per l'Ingegneria, Universit\`{a} degli studi di Roma Sapienza,
}
   \centerline{Via A. Scarpa 16,}
   \centerline{00161 Roma, Italy.}
}

\bigskip

\begin{abstract}
\noindent We prove a generalized version of Friedrichs and Gaffney inequalities for a bounded $(\varepsilon,\delta)$ domain $\Omega\subset\mathbb{R}^n$, $n=2,3$, by adapting the methods of Jones to our framework.
\end{abstract}

\medskip

\noindent\textbf{Keywords:} Friedrichs inequality, Gaffney inequality, $(\varepsilon,\delta)$ domains, Whitney decomposition, coercivity estimates.\\

\noindent{\textbf{2010 Mathematics Subject Classification:} Primary: 35A23. Secondary: 35Q61, 78A25.}

\bigskip

\section{Introduction}
\setcounter{equation}{0}

\noindent The aim of this paper is to prove Friedrichs inequality for $(\varepsilon,\delta)$ domains. This inequality has been introduced in different frameworks by Friedrichs \cite{friedrichs} and Gaffney \cite{gaffney}, and in the literature it is known with different names according to the setting where it is used. In the study of Maxwell problems or Navier-Stokes equations, this inequality is a key tool to prove the coercivity of the associated energy forms. From the point of view of applications, it is interesting to study vector BVPs in irregular domains (see e.g. \cite{CHLTV,LVstokes}) and their numerical approximation, hence it is crucial to extend these inequalities to the case of suitable irregular sets. From this perspective, we confine ourselves to two or three-dimensional domains.\\
Gaffney inequality can be deduced from the Friedrichs inequality. To our knowledge, such inequalities hold for convex and Lipschitz domains; among the others, we refer to \cite{bauerpauly,Schw16,NPW15,amrouche}, see also \cite{dacorogna} and the references listed in. In this paper, we first prove Friedrichs inequality for $(\varepsilon,\delta)$ domains, and then prove Gaffney inequality by adapting the methods of \cite{duran} (developed for Korn inequality) to this framework.\\
The class of $(\varepsilon,\delta)$ domains has been introduced by Jones \cite{Jones}, and it is quite general, since the boundary of an $(\varepsilon,\delta)$ domain can be highly non-rectifiable, e.g. fractal or a $d$-set (see Definitions \ref{defepsdelta} and \ref{dset}).\\
In the literature, for $\Omega\subset\R^n$ $(n=2,3)$ sufficiently smooth, the Friedrichs inequality reads as follows: if $v\in W^{1,p}(\Omega)^n$, there exists a positive constant $C$, depending on $\Omega$, $n$ and $p$, such that
\begin{equation}\label{Fr1}
\|v\|_{W^{1,p}(\Omega)^n}\leq C(\|v\|_{L^p(\Omega)^n}+\|\dive v\|_{L^p(\Omega)}+\|\curl v\|_{L^p(\Omega)^n}).
\end{equation}
Gaffney inequality is a direct consequence of Friedrichs inequality \eqref{Fr1} when considering boundary conditions. We introduce the following spaces:
\[W^p(\dive,\Omega):=\left\lbrace u\in L^p(\Omega)^n\,:\,\dive u\in L^p(\Omega)  \right\rbrace,\]
\[W^p_0(\dive,\Omega):=\{u\in W^p(\dive,\Omega)\,:\,\nu\cdot u=0\,\,\text{on}\;\partial\Omega \},\]
\[W^p(\curl,\Omega):=\left\lbrace u\in L^p(\Omega)^n\,:\,\curl u\in L^p(\Omega)^n  \right\rbrace,\]
\[W^p_0(\curl,\Omega):=\{u\in W^p(\curl,\Omega)\,:\,\nu\times u=0\,\,\text{on}\;\partial\Omega \},\]
where $\cdot$ and $\times$ denote respectively the usual scalar and cross products between vectors in $\R^n$. The boundary conditions have to be interpreted in a suitable weak sense (see e.g. \cite{temam}).\\
When $v\in W^p(\dive,\Omega)\cap W^p_0(\curl,\Omega)$ or $v\in W^p(\curl,\Omega)\cap W^p_0(\dive,\Omega)$, Gaffney inequality takes the following form:
\begin{equation}\label{Fr2}
\|\nabla v\|_{L^p(\Omega)^{\enne}}\leq C(\|\dive v\|_{L^p(\Omega)}+\|\curl v\|_{L^p(\Omega)^n}).
\end{equation}
Our aim is to extend Gaffney inequality to those $(\varepsilon,\delta)$ domains for which it is possible to give an interpretation of the boundary conditions. In particular, we consider $(\varepsilon,\delta)$ domains $\Omega$ in $\R^n$ whose boundaries are $d$-sets or arbitrary closed sets in the sense of Jonsson \cite{jonsson91}. In these cases, it can be proved that the spaces $W^p_0(\dive,\Omega)$ and $W^p_0(\curl,\Omega)$ are well defined because generalized Green and Stokes formulas hold. This implies that the normal and tangential traces are well defined as elements of the duals of suitable trace Besov spaces on the boundary (see \cite{LaVe2} and \cite{CHLTV}).

We extend \eqref{Fr1} and \eqref{Fr2} to $(\varepsilon,\delta)$ domains $\Omega\subset\R^n$ for either $v\in W^p(\curl,\Omega)\cap W^p_0(\dive,\Omega)$ (see Section \ref{divergenza}) or $v\in W^p(\dive,\Omega)\cap W^p_0(\curl,\Omega)$ (see Section \ref{rotore}), according to the boundary conditions under consideration. The main results of this paper are Theorems \ref{fridis}, \ref{gaffin}, \ref{fridisrot} and \ref{gaffinrot}.\\
The proof of our results deeply relies on the assumptions on $\Omega$. Since $\Omega$ is an $(\varepsilon,\delta)$ domain, for each $v\in W^p(\curl,\Omega)\cap W^p_0(\dive,\Omega)$ (for each $v\in W^p(\dive,\Omega)\cap W^p_0(\curl,\Omega)$ respectively) we construct a suitable extension $Ev$ by adapting Jones' approach \cite{Jones}. More precisely, we consider a Whitney decomposition of $\Omega$ and we construct an extension operator in terms of suitable linear polynomials which satisfies the crucial estimates \eqref{corol1} and \eqref{corol2} (\eqref{corol1rot} and \eqref{corol2rot} respectively). The thesis is then achieved by density arguments.\\
Throughout the paper, $C$ will denote different positive constant. Sometimes, we indicate the dependence of these constants on some particular parameters in parentheses.

\section{$(\varepsilon,\delta)$ domains and trace results}
\setcounter{equation}{0}

We recall the definition of $(\varepsilon,\delta)$ (or Jones) domain.
\begin{definition}\label{defepsdelta} Let $\mathcal{F}\subset\R^n$ be open and connected and $\mathcal{F}^c:=\R^n\setminus\mathcal{F}$. For $x\in\mathcal{F}$, let $\displaystyle d(x):=\inf_{y\in\mathcal{F}^c}|x-y|$. We say say that $\mathcal{F}$ is an $(\varepsilon,\delta)$ domain if, whenever $x,y\in\mathcal{F}$ with $|x-y|<\delta$, there exists a rectifiable arc $\gamma\in\mathcal{F}$ joining $x$ to $y$ such that
\begin{center}
$\displaystyle\ell(\gamma)\leq\frac{1}{\varepsilon}|x-y|\quad$ and\quad $\displaystyle d(z)\geq\frac{\varepsilon|x-z||y-z|}{|x-y|}$ for every $z\in\gamma$.
\end{center}
\end{definition}

As pointed out in the Introduction, we consider two particular classes of $(\varepsilon,\delta)$ domains $\Omega\subset\R^n$:
\begin{itemize}
	\item[$i)$] $(\varepsilon,\delta)$ domains having as boundary a $d$-set;
	\item[$ii)$] arbitrary closed $(\varepsilon,\delta)$ domains in the sense of \cite{jonsson91}.
\end{itemize}
For the sake of completeness, we recall the definition of $d$-set given in \cite{JoWa}.
\begin{definition}\label{dset}
A closed nonempty set $\mathcal{M}\subset\R^n$ is a $d$-set (for $0<d\leq n$) if there exist a Borel measure $\mu$ with $\supp\mu=\mathcal{M}$ and two positive constants $c_1$ and $c_2$ such that
\begin{equation}\label{defindset}
c_1r^{d}\leq \mu(B(P,r)\cap\mathcal{M})\leq c_2 r^{d}\quad\forall\,P \in\mathcal{M}.
\end{equation}
The measure $\mu$ is called $d$-measure.
\end{definition}

\medskip

In both the cases $i)$ and $ii)$, we can prove trace theorems, i.e. Green and Stokes formulas. For the sake of simplicity, we restrict ourselves to the case in which $\partial\Omega$ is a $d$-set. We recall the definition of Besov space specialized to our case. For generalities on Besov spaces, we refer to \cite{JoWa}.
\begin{definition}
Let $\mathcal{G}$ be a $d$-set with respect to a $d$-measure $\mu$  and $\alpha=1-\frac{n-d}{p}$. ${B^{p,p}_\alpha(\mathcal{G})}$ is the space of functions for which the following norm is finite:
$$
\|u\|_{B^{p,p}_\alpha(\mathcal{G})}=\|u\|_{L^p(\mathcal{G})}+\left(\quad\iint_{|P-P'|<1}\frac{|u(P)-u(P')|^p}{|P-P'|^{d+p\alpha}}\,\de\mu(P)\,\de\mu(P')\right)^\frac{1}{p}.
$$
\end{definition}

Throughout the paper, $p'$ will denote the H\"older conjugate exponent of $p$. In the following, we denote the dual of the Besov space on a $d$-set $\mathcal{G}$ with $(B^{p,p}_\alpha(\mathcal{G}))'$; this space coincides with the space $B^{p',p'}_{-\alpha}(\mathcal{G})$ (see \cite{JoWa2}).

\begin{theorem}[Stokes formula]\label{stokes}
Let $u\in W^p(\curl,\Omega)$. There exists a linear and continuous operator $l_\tau(u)=u\times\nu$ from $W^p(\curl,\Omega)$ to $((B^{p',p'}_\alpha(\partial\Omega))')^3$.

The following generalized Stokes formula holds for every $v\in W^{1,p'}(\Omega)^n$:
\begin{equation}\label{stokesformula}
\left\langle u\times\nu,v\right\rangle_{((B^{p',p'}_\alpha(\partial\Omega))')^3, B^{p',p'}_\alpha(\partial\Omega)^3}
=\int_\Omega u\cdot \curl v\,\de x +\int_\Omega v\cdot \curl u\,\de x.
\end{equation}

Moreover, the operator $u \mapsto l_\tau(u)=u\times\nu$ is linear and continuous on $B^{p',p'}_\alpha(\partial\Omega)^3$.


\end{theorem}

\begin{theorem}[Green formula]\label{green}
Let $u\in W^p(\dive,\Omega)$. There exists a linear and continuous operator $l_\nu(u)=u\cdot\nu$ from $W^p(\dive,\Omega)$ to $(B^{p',p'}_\alpha(\partial\Omega))'$.

The following generalized Green formula holds for every $v\in W^{1,p'}(\Omega)$:
\begin{equation}\label{greenformula}
\left\langle u\cdot\nu,v\right\rangle_{(B^{p',p'}_\alpha(\partial\Omega))', B^{p',p'}_\alpha(\partial\Omega)}
=\int_\Omega u\cdot\nabla v\,\de x +\int_\Omega v\dive u\,\de x.
\end{equation}

Moreover, the operator $u \mapsto l_\nu(u)=u\cdot\nu$ is linear and continuous on $B^{p',p'}_\alpha(\partial\Omega)$.


\end{theorem}

\medskip

For the proofs we refer the reader to \cite{LaVe2} and \cite{CHLTV} with small suitable changes. Examples of domains for which Theorems \ref{stokes} and \ref{green} hold are 2D or 3D Koch-type domains. Formulas \eqref{stokesformula} and \eqref{greenformula} give a rigorous meaning of the boundary conditions in $W^p_0(\curl,\Omega)$ and $W^p_0(\dive,\Omega)$ respectively in terms of the dual of suitable Besov spaces.

\section{Friedrichs and Gaffney inequalities}
\setcounter{equation}{0}

From now on, let $\Omega\subset\R^n$ be a bounded $(\varepsilon,\delta)$ domain, for $n=2,3$, having as boundary $\partial\Omega$ a $d$-set.

\subsection{The case $v\in W^p(\curl,\Omega)\cap W^p_0(\dive,\Omega)$}\label{divergenza}

We first consider the case $v\in W^p(\curl,\Omega)\cap W^p_0(\dive,\Omega)$. We point out that, since $\nu\cdot v=0$ on $\partial\Omega$, we have that
\begin{equation}\label{mediadive}
\int_\Omega \dive v\,\de x=0.
\end{equation}

Let $S\subset\R^n$ be a measurable subset of $\R^n$; we denote by $\bar x$ its barycenter.\\
We construct the affine vector field $P_S(u)$ associated to $S$ and $u\in W^p(\curl,S)\cap W^p_0(\dive,S)$ in the following way:
\begin{equation}\label{Pdef}
P_S(u)(x)=a+B(x-\bar x),
\end{equation}
where $a\in\R^n$ and $B$ is a $n\times n$ matrix with entries $b_{ij}$ defined as
\begin{equation}\label{definizioni}
a=\frac{1}{|S|}\int_S u\,\de x\quad\text{and}\quad b_{ij}=\frac{1}{2|S|}\int_S \left(\frac{\partial u_i}{\partial x_j}+\frac{\partial u_j}{\partial x_i}\right)\,\de x.
\end{equation}

We point out that, from the definition, $B$ is a symmetric matrix. Moreover, by calculation it follows that $\curl(P_S(u))=0$,
\begin{equation}\label{diveP}
\dive (P_S(u))=\frac{1}{|S|}\int_S \dive u\,\de x
\end{equation}
and
\begin{equation}\label{prop2}
\int_S (u-P_S(u))\,\de x=0.
\end{equation}

By direct computation, it holds that
\begin{equation}\label{stimagradiente}
\|\nabla(u-P_S(u))\|_{L^p(S)^{\enne}}\leq C\|\nabla u\|_{L^p(S)^{\enne}},
\end{equation}
where $C$ depends only on $|S|$.

Let us now suppose that \eqref{Fr2} holds in $S$ for $u\in W^p(\curl,S)\cap W^p_0(\dive,S)$. \eqref{stimagradiente} infers that
\begin{equation}\label{int1}
\|\nabla(u-P_S(u))\|_{L^p(S)^{\enne}}\leq C\left(\|\curl u\|_{L^p(S)^n}+\|\dive u\|_{L^p(S)}\right).
\end{equation}
Since $u-P_S(u)$ has vanishing mean value on $S$ from \eqref{prop2}, from Poincar\'e-Wirtinger inequality and \eqref{int1} we have
\begin{equation}\label{poincare}
\|u-P_S(u)\|_{L^p(S)^n}\leq C \diam(S)\left(\|\curl u\|_{L^p(S)^n}+\|\dive u\|_{L^p(S)}\right),
\end{equation}
where $\diam(S)$ is the diameter of $S$. Now, one can easily see that
\begin{equation}\label{stimaPinf}
\|\nabla P_S(u)\|_{L^\infty(S)^{\enne}}\leq\|\nabla u\|_{L^\infty(S)^{\enne}};
\end{equation}
hence, by using again Poincar\'e-Wirtinger inequality (with $p=\infty$), triangle inequality and \eqref{stimaPinf} we get
\begin{equation}\label{stimainf}
\|u-P_S(u)\|_{L^\infty(S)^n}\leq C \diam(S)\|\nabla(u-P_S(u))\|_{L^\infty(S)^{\enne}}\leq 2C\diam(S)\|\nabla u\|_{L^\infty(S)^{\enne}}.
\end{equation}
From now on, we choose $v\in W^{1,\infty}(\Omega)^n$. The thesis will then follow by density arguments. We construct the extension $Ev$ following the approach of Jones \cite{Jones} by using the linear polynomials $P_S(v)$.\\
Let us recall that any open set $\Omega\subset\R^n$ admits a so-called \emph{Whitney decomposition} (see \cite{whitney}, \cite{stein}) into dyadic cubes $S_k$, i.e.
\begin{center}
$\displaystyle\Omega=\bigcup_k S_k$.
\end{center}
This decomposition is such that
\begin{equation}\label{w1}
1\leq\frac{{\rm dist}(S_k,\partial\Omega)}{\ell(S_k)}\leq 4\sqrt{n}\quad\forall\,k,
\end{equation}
\begin{equation}\label{w2}
S_j^0\cap S_k^0=\emptyset\quad\text{if}\,\,j\neq k,
\end{equation}
\begin{equation}\label{w3}
\frac{1}{4}\leq\frac{\ell(S_j)}{\ell(S_k)}\leq 4\quad\text{if}\,\,S_j\cap S_k\neq\emptyset,
\end{equation}
where $S^0$ denotes the interior of $S$ and $\ell(S)$ is the edgelength of a cube $S$.\\
Let now $W_1=\{S_k\}$ be a Whitney decomposition of $\Omega$ and $W_2=\{Q_j\}$ be a Whitney decomposition of $(\Omega^c)^0$. We set
\begin{equation*}
W_3=\left\{Q_j\in W_2\,:\,\ell(Q_j)\leq\frac{\varepsilon\delta}{16n}\right\}.
\end{equation*}
In his paper, Jones has shown that, for every $Q_j\in W_3$, one can choose a $\lq\lq$reflected" cube $Q_j^*=S_k\in W_1$ such that
\begin{equation}\label{numero}
1\leq\frac{\ell(S_k)}{\ell(Q_j)}\leq 4\quad\text{and}\quad{\rm dist}(Q_j,S_k)\leq C\ell(Q_j),
\end{equation}
see Lemma 2.4 and Lemma 2.8 in \cite{Jones}. Moreover, if $Q_j,Q_k\in W_3$ have non-empty intersection, there exists a chain $F_{j,k}=\{Q_j^*=S_1,S_2,\dots,S_m=Q_k^*\}$ of cubes in $W_1$ which connects $Q_j^*$ and $Q_k^*$ such that $S_i\cap S_{i+1}\neq\emptyset$ and $m\leq C(\varepsilon,\delta)$.\\
From \cite{stein}, \cite{whitney} it follows that there exists a partition of unity $\{\phi_j\}$, associated with the Whitney decomposition,  such that
\begin{center}
$\phi_j\in C^\infty(\R^n),\quad\supp\phi_j\subset\frac{17}{16}Q_j,\quad 0\leq\phi_j\leq 1$,\\
$\sum_{Q_j\in W_3}\phi_j=1\,\,\text{on}\,\,\bigcup_{Q_j\in W_3} Q_j\quad$ and $\quad|\nabla\phi_j|\leq C\ell(Q_j)^{-1}\quad\forall\,j$.
\end{center}

\medskip

For $v\in W^{1,\infty}(\Omega)^n$, let $P_j:=P_{Q_j^*}(v)$ be defined as in \eqref{Pdef} and \eqref{definizioni}. We now define the extension $Ev$ of $v$ to $\R^n$ in the following way:
\begin{equation*}
Ev=
\begin{cases}
\displaystyle\sum_{Q_j\in W_3} P_j\phi_j\quad &\text{in}\,\,(\Omega^c)^0,\\[2mm]
v &\text{in}\,\,\Omega.
\end{cases}
\end{equation*}
We point out that, since the boundary of an $(\varepsilon,\delta)$ domain has zero measure (see Lemma 2.3 in \cite{Jones}), it follows that $Ev$ is defined a.e. in $\R^n$.\\
From now on, if not otherwise specified, in this subsection we assume that $v\in W^p(\curl,\Omega)\cap W^p_0(\dive,\Omega)\cap W^p_0(\dive,S)$ for every $S\in W_1$. We now prove some preliminary lemmas. For the sake of completeness, we recall Lemma 2.1 in \cite{Jones}.
\begin{lemma}\label{lemmajones} Let $Q$ be a cube and let $F,G\subset Q$ be two measurable subsets such that $|F|,|G|\geq\gamma|Q|$ for some $\gamma>0$. If $P$ is a polynomial of degree 1, then
\begin{equation*}
\|P\|_{L^p(F)}\leq C(\gamma)\|P\|_{L^p(G)}.
\end{equation*}
\end{lemma}

\begin{lemma}\label{lemma3.2} Let $F=\{S_1,\dots,S_m\}$ be a chain of cubes in $W_1$. Then
\begin{equation}\label{stimaPLp}
\|P_{S_1}(v)-P_{S_m}(v)\|_{L^p(S_1)^n}\leq C(m)\ell(S_1)\left(\|\curl v\|_{L^p(\cup_j S_j)^n}+\|\dive v\|_{L^p(\cup_j S_j)}\right)
\end{equation}
and
\begin{equation}\label{stimaPLinf}
\|P_{S_1}(v)-P_{S_m}(v)\|_{L^\infty(S_1)^n}\leq C(m)\ell(S_1)\|\nabla v\|_{L^\infty(\cup_j S_j)^{\enne}}.
\end{equation}
\end{lemma}

\begin{proof} We will use \eqref{poincare}, where $S$ is a cube or a union of two neighboring cubes. From \eqref{w3}, it follows that the number of possible geometries of $S$ is finite; hence, we can find a uniform constant in \eqref{poincare}.\\
By using Lemma \ref{lemmajones}, we get
\begin{align*}
&\|P_{S_1}(v)-P_{S_m}(v)\|_{L^p(S_1)^n}\leq\sum_{r=1}^{m-1}\|P_{S_r}(v)-P_{S_{r+1}}(v)\|_{L^p(S_1)^n}\\[2mm]
&\leq c(m)\sum_{r=1}^{m-1}\|P_{S_r}(v)-P_{S_{r+1}}(v)\|_{L^p(S_r)^n}\\[2mm]
&\leq c(m)\sum_{r=1}^{m-1}\left\{\|P_{S_r}(v)-P_{S_r\cup S_{r+1}}(v)\|_{L^p(S_r)^n}+\|P_{S_r\cup S_{r+1}}(v)-P_{S_{r+1}}(v)\|_{L^p(S_{r+1})^n}\right\}\\[2mm]
&\leq c(m)\sum_{r=1}^{m-1}\left\{\|P_{S_r}(v)-v\|_{L^p(S_r)^n}+\|P_{S_{r+1}}(v)-v\|_{L^p(S_{r+1})^n}+\|P_{S_r\cup S_{r+1}}(v)-v\|_{L^p(S_r\cup S_{r+1})^n}\right\}\\[2mm]
&\leq Cc(m)\ell(S_1)\left(\|\curl v\|_{L^p(\cup_j S_j)^n}+\|\dive v\|_{L^p(\cup_j S_j)}\right),
\end{align*}
where we used the fact that $F$ is a chain, integral properties and finally \eqref{poincare}.\\
The proof of \eqref{stimaPLinf} follows analogously by using \eqref{stimainf}.
\end{proof}

For every $Q_j,Q_k\in W_3$ with non-empty intersection, we now choose a chain $F_{j,k}$ which connects $Q_j^*$ and $Q_k^*$ and such that $m\leq C(\varepsilon,\delta)$. We define
\begin{equation*}
F(Q_j)=\bigcup_{Q_k\in W_3, Q_j\cap Q_k\neq\emptyset} F_{j,k};
\end{equation*}
hence
\begin{equation}\label{finito}
\left\|\sum_{Q_k\,:\,Q_j\cap Q_k\neq\emptyset} \chi_{\cup F_{j,k}}\right\|_{L^\infty(\R^n)}\leq C\quad\forall\,Q_j\in W_3.
\end{equation}
We now prove two lemmas which allow us to control the norms of $Ev$, $\dive (Ev)$, $\curl (Ev)$ and $\nabla (Ev)$ in $(\Omega^c)^0$.

\begin{lemma}\label{lemma1} Let $Q_0\in W_3$. We have that:
\begin{equation}\label{stima1}
\|Ev\|_{L^p(Q_0)^n}\leq C\left(\|v\|_{L^p(Q_0^*)^n}+\ell(Q_0)(\|\curl v\|_{L^p(F(Q_0))^n}+\|\dive v\|_{L^p(F(Q_0))})\right),
\end{equation}
\begin{equation}\label{stima2}
\|\curl(Ev)\|_{L^p(Q_0)^n}+\|\dive (Ev)\|_{L^p(Q_0)}\leq C\left(\|\curl v\|_{L^p(F(Q_0))^n}+\|\dive v\|_{L^p(F(Q_0))}\right),
\end{equation}
\begin{equation}\label{stima3}
\|Ev\|_{L^\infty(Q_0)^n}\leq C\left(\|v\|_{L^\infty(Q_0^*)^n}+\ell(Q_0)\|\nabla v\|_{L^\infty(F(Q_0))^{\enne}}\right),
\end{equation}
\begin{equation}\label{stima4}
\|\nabla(Ev)\|_{L^\infty(Q_0)^{\enne}}\leq C\|\nabla v\|_{L^\infty(F(Q_0))^{\enne}}.
\end{equation}

\end{lemma}

\begin{proof} We recall that, from the definition of $Ev$, on $Q_0$ we have that $\displaystyle Ev=\sum_{Q_j\in W_3} P_j\phi_j$. Moreover, since $\displaystyle\sum_{Q_j\in W_3}\phi_j\equiv 1$ on $\displaystyle\bigcup_{Q_j\in W_3} Q_j$, we get
\begin{equation*}
\left\|\sum_{Q_j\in W_3}P_j\phi_j\right\|_{L^p(Q_0)^n}\leq\|P_0\|_{L^p(Q_0)^n}+\left\|\sum_{Q_j\in W_3}(P_j-P_0)\phi_j\right\|_{L^p(Q_0)^n}:=A+B.
\end{equation*}
We now estimate $A$ and $B$ separately. As to $A$, from Lemma \ref{lemmajones} and \eqref{poincare}, we get
\begin{align}\label{stimaAprima}
A&=\|P_0\|_{L^p(Q_0)^n}\leq C\|P_0\|_{L^p(Q_0^*)^n}\leq C(\|P_0-v\|_{L^p(Q_0^*)^n}+\|v\|_{L^p(Q_0^*)^n})\notag\\[2mm]
&\leq C(\ell(Q_0)(\|\curl v\|_{L^p(Q_0^*)^n}+\|\dive v\|_{L^p(Q_0^*)})+\|v\|_{L^p(Q_0^*)^n}),
\end{align}
where we estimated $\ell(Q_0^*)$ with $\ell(Q_0)$ using \eqref{numero}, since $Q_0\in W_3$. We point out that, thanks to \eqref{finito}, the norms in the right-hand side of \eqref{stimaAprima} can be estimated in terms of the $L^p(\cup_j F_{0,j})$-norms. Hence, we get the following:
\begin{equation}\label{stimaA}
A\leq C(\ell(Q_0)(\|\curl v\|_{L^p(\cup_j F_{0,j})^n}+\|\dive v\|_{L^p(\cup_j F_{0,j})})+\|v\|_{L^p(\cup_j F_{0,j})^n}).
\end{equation}
As to $B$, from the properties of $\phi_j$ it is sufficient to bound $\|P_j-P_0\|_{L^p(Q_0)^n}$. By using again Lemma \ref{lemmajones}, \eqref{stimaPLp} and proceeding as above, we get
\begin{equation}\label{stimaB}
B\leq\|P_j-P_0\|_{L^p(Q_0)^n}\leq C\|P_j-P_0\|_{L^p(Q_0^*)^n}\leq C\ell(Q_0)(\|\curl v\|_{L^p(\cup_j F_{0,j})^n}+\|\dive v\|_{L^p(\cup_j F_{0,j})}).
\end{equation}
Hence from \eqref{stimaA} and \eqref{stimaB} we get \eqref{stima1}. Estimate \eqref{stima3} follows similarly by using \eqref{stimainf} and \eqref{stimaPLinf}.\\
We now remark that, on $Q_0$, we have that
\begin{equation*}
Ev=\sum_{Q_j\in W_3} P_j\phi_j=P_0\sum_{Q_j\in W_3}\phi_j+\sum_{Q_j\in W_3} (P_j-P_0)\phi_j=P_0+\sum_{Q_j\in W_3} (P_j-P_0)\phi_j.
\end{equation*}
Therefore, since $\curl(P_0)=0$, we have that
\begin{equation*}
\curl(Ev)=\sum_{Q_j\in W_3}\curl((P_j-P_0)\phi_j).
\end{equation*}
Moreover, from \eqref{mediadive} and \eqref{diveP} it follows that
\begin{equation*}
\dive(Ev)=\sum_{Q_j\in W_3}\dive((P_j-P_0)\phi_j).
\end{equation*}
Since there is a finite number of cubes $Q_j$ such that $\phi_j\neq 0$ in $Q_0$ and having non-empty intersection with $Q_0$, from \eqref{w3} we have that $\ell(Q_j)\geq\frac{1}{4}\ell(Q_0)$. From the properties of $\phi_j$, this implies that $|\nabla\phi_j|\leq\frac{C}{4}\ell(Q_0)^{-1}$.\\
By using vector identities, Lemma \ref{lemmajones} and \eqref{stimaPLp}, we have that
\begin{align*}
&\|\curl((P_j-P_0)\phi_j)\|_{L^p(Q_0)^n}=\|(P_j-P_0)\times\nabla\phi_j\|_{L^p(Q_0)^n}\leq C\|P_j-P_0\|_{L^p(Q_0)^n}\|\nabla\phi_j\|_{L^p(Q_0)^n}\\[2mm]
&\leq C\ell(Q_0)^{-1}\|P_j-P_0\|_{L^p(Q_0)^n}\leq C\ell(Q_0)^{-1}\|P_j-P_0\|_{L^p(Q_0^*)^n}\\[2mm]
&\leq C(\|\curl v\|_{L^p(\cup_j F_{0,j})^n}+\|\dive v\|_{L^p(\cup_j F_{0,j})}).
\end{align*}
As to divergence term, similarly as above we get
\begin{align*}
&\|\dive((P_j-P_0)\phi_j)\|_{L^p(Q_0)}=\|(P_j-P_0)\cdot\nabla\phi_j\|_{L^p(Q_0)}\leq\|P_j-P_0\|_{L^p(Q_0)^n}\|\nabla\phi_j\|_{L^p(Q_0)^n}\\[2mm]
&\leq C(\|\curl v\|_{L^p(\cup_j F_{0,j})^n}+\|\dive v\|_{L^p(\cup_j F_{0,j})}).
\end{align*}
Summing up in $j$ we get
\begin{equation*}
\|\curl(Ev)\|_{L^p(Q_0)^n}+\|\dive(Ev)\|_{L^p(Q_0)}\leq C(\|\curl v\|_{L^p(F(Q_0))^n}+\|\dive v\|_{L^p(F(Q_0))}),
\end{equation*}
i.e. \eqref{stima2}.\\
We are left to prove \eqref{stima4}. Similarly as above, we have that
\begin{equation*}
\nabla(Ev)=\nabla P_0+\sum_{Q_j\in W_3} \nabla((P_j-P_0)\phi_j).
\end{equation*}
From Lemma \ref{lemmajones} and \eqref{stimaPinf}, we get
\begin{equation*}
\|\nabla P_0\|_{L^\infty(Q_0)^{\enne}}\leq C\|\nabla P_0\|_{L^\infty(Q_0^*)^{\enne}}\leq C\|\nabla v\|_{L^\infty(Q_0^*)^{\enne}}\leq C\|\nabla v\|_{L^\infty(\cup_j F_{0,j})^{\enne}}.
\end{equation*}
As above, it follows that
\begin{align*}
&\|\nabla(P_j-P_0)\|_{L^\infty(Q_0)^{\enne}}\leq C\|\nabla(P_j-P_0)\|_{L^\infty(Q_0^*)^{\enne}}\leq C\|\nabla(P_j-P_0)\|_{L^\infty(Q_0^*\cup Q_j^*)^{\enne}}\\[2mm]
&\leq C\|\nabla v\|_{L^\infty(Q_0^*\cup Q_j^*)^{\enne}}\leq C\|\nabla v\|_{L^\infty(\cup_j F_{0,j})^{\enne}}
\end{align*}
From these inequalities, \eqref{stima4} follows and the proof is complete.
\end{proof}
We now prove a result similar to Lemma \ref{lemma1}, which relates to the cubes of $(\Omega^c)^0$ not belonging to $W_3$.

\begin{lemma}\label{lemma2} Let $Q_0\in W_2\setminus W_3$. We have that:
\begin{equation}\label{stima1comp}
\|Ev\|_{L^p(Q_0)^n}\leq C\sum_{Q_j\in W_3\,:\,Q_j\cap Q_0\neq\emptyset}\left(\|v\|_{L^p(Q_j^*)^n}+\|\curl v\|_{L^p(Q_j^*)^n}+\|\dive v\|_{L^p(Q_j^*)}\right),
\end{equation}
\begin{align}
\|\curl(Ev)\|_{L^p(Q_0)^n}&+\|\dive(Ev)\|_{L^p(Q_0)}\notag\\[2mm]
&\leq C\sum_{Q_j\in W_3\,:\,Q_j\cap Q_0\neq\emptyset}\left(\|v\|_{L^p(Q_j^*)^n}+\|\curl v\|_{L^p(Q_j^*)^n}+\|\dive v\|_{L^p(Q_j^*)}\right),\label{stima2comp}
\end{align}
\begin{equation}\label{stima3comp}
\|Ev\|_{L^\infty(Q_0)^n}\leq C\sum_{Q_j\in W_3\,:\,Q_j\cap Q_0\neq\emptyset}\left(\|v\|_{L^\infty(Q_j^*)^n}+\|\nabla v\|_{L^\infty(Q_j^*)^{\enne}}\right),
\end{equation}
\begin{equation}\label{stima4comp}
\|\nabla(Ev)\|_{L^\infty(Q_0)^{\enne}}\leq C\sum_{Q_j\in W_3\,:\,Q_j\cap Q_0\neq\emptyset}\left(\|v\|_{L^\infty(Q_j^*)^n}+\|\nabla v\|_{L^\infty(Q_j^*)^{\enne}}\right).
\end{equation}

\end{lemma}

\begin{proof} We start by pointing out that, if $\phi_j\neq 0$ on $Q_0$, we have $Q_j\cap Q_0\neq\emptyset$ (since $\supp\phi_j\subset\frac{17}{16}Q_j)$. Therefore, since $Q_0\in W_2\setminus W_3$, we have
\begin{equation}\label{stimalunghezza}
\ell (Q_j)\geq\frac{1}{4}\ell (Q_0)\geq\frac{\varepsilon\delta}{64n}.
\end{equation}
On $Q_0$ we have that
\begin{equation*}
|Ev|=\left|\sum_{Q_j\in W_3\,:\,Q_j\cap Q_0\neq\emptyset} P_j\phi_j\right|\leq \sum_{Q_j\in W_3\,:\,Q_j\cap Q_0\neq\emptyset} |P_j|.
\end{equation*}
From Lemma \ref{lemmajones} and triangle inequality, we get
\begin{equation}\label{interm}
\|P_j\|_{L^p(Q_0)^n}\leq C\|P_j\|_{L^p(Q_j)^n}\leq C\|P_j\|_{L^p(Q_j^*)^n}\leq C(\|P_j-v\|_{L^p(Q_j^*)^n}+\|v\|_{L^p(Q_j^*)^n}).
\end{equation}

From \eqref{interm} and \eqref{poincare}, it follows that
\begin{align*}
&\|Ev\|_{L^p(Q_0)^n}\leq\sum_{Q_j\in W_3\,:\,Q_j\cap Q_0\neq\emptyset}\|P_j\|_{L^p(Q_0)^n}\\[3mm]
&\leq C\sum_{Q_j\in W_3\,:\,Q_j\cap Q_0\neq\emptyset}\left(\|v\|_{L^p(Q_j^*)^n}+\diam(Q_j^*)(\|\curl v\|_{L^p(Q_j^*)^n}+\|\dive v\|_{L^p(Q_j^*)})\right).
\end{align*}
Sine $\Omega$ is bounded, we can estimate $\diam(Q_j^*)$ with a constant depending on $\diam(\Omega)$, thus proving \eqref{stima1comp}.\\
We come to \eqref{stima2comp}. By proceeding as in the proof of Lemma \ref{lemma1} and by using \eqref{stimalunghezza}, the following estimate holds:
\bigskip
\begin{align*}
&\|\curl(Ev)\|_{L^p(Q_0)^n}+\|\dive(Ev)\|_{L^p(Q_0)}=\left\|\sum_{Q_j\in W_3\,:\,Q_j\cap Q_0\neq\emptyset}\curl(P_j\phi_j)\right\|_{L^p(Q_0)^n}\\[2mm]
&+\left\|\sum_{Q_j\in W_3\,:\,Q_j\cap Q_0\neq\emptyset}\dive(P_j\phi_j)\right\|_{L^p(Q_0)}\leq C\sum_{Q_j\in W_3\,:\,Q_j\cap Q_0\neq\emptyset}\|P_j\|_{L^p(Q_0)^n}\|\nabla\phi_j\|_{L^p(Q_0)^{\enne}}\\[2mm]
&\leq C\ell(Q_0)^{-1}\sum_{Q_j\in W_3\,:\,Q_j\cap Q_0\neq\emptyset}\|P_j\|_{L^p(Q_0)^n}\leq C\left(\frac{\varepsilon\delta}{64n}\right)^{-1}\sum_{Q_j\in W_3\,:\,Q_j\cap Q_0\neq\emptyset}\|P_j\|_{L^p(Q_0^*)^n}.
\end{align*}
By proceeding as above, we get \eqref{stima2comp}. Estimates \eqref{stima3comp} and \eqref{stima4comp} follow in a similar way by using \eqref{stimainf} and \eqref{stimaPinf}.
\end{proof}

From the above lemmas we obtain the following result.

\begin{prop}\label{proposiz} For every $v\in W^{1,\infty}(\Omega)^n$ such that $v\in W^p(\curl,\Omega)\cap W^p_0(\dive,\Omega)$ we have
\begin{align}
\|Ev\|_{L^p((\Omega^c)^0)^n}&+\|\dive(Ev)\|_{L^p((\Omega^c)^0)}+\|\curl(Ev)\|_{L^p((\Omega^c)^0)^n}\notag\\[2mm]
&\leq C\left(\|v\|_{L^p(\Omega)^n}+\|\dive v\|_{L^p(\Omega)}+\|\curl v\|_{L^p(\Omega)^n}\right)\label{corol1}
\end{align}
and
\begin{equation}\label{corol2}
\|Ev\|_{W^{1,\infty}((\Omega^c)^0)^n}\leq C\|v\|_{W^{1,\infty}(\Omega)^n}.
\end{equation}
\end{prop}

\begin{proof} By summing up over every $Q_0\in W_2$, the thesis follows as a direct consequence of Lemma \ref{lemma1} and Lemma \ref{lemma2}. In particular, \eqref{corol1} follows from \eqref{stima1}, \eqref{stima2}, \eqref{stima1comp} and \eqref{stima2comp}, while \eqref{corol2} follows from \eqref{stima3}, \eqref{stima4}, \eqref{stima3comp} and \eqref{stima4comp}.
\end{proof}

We now prove the first main result of this paper, which follows from the above lemmas.

\begin{theorem}[Friedrichs inequality]\label{fridis} Let $\Omega\subset\R^n$ be a bounded $(\varepsilon,\delta)$ domain with $\partial\Omega$ a $d$-set. There exists a constant $C=C(\varepsilon,\delta,n,p,\Omega)>0$ such that, for every $v\in W^{1,p}(\Omega)^n$ such that $v\in W^p(\curl,\Omega)\cap W^p_0(\dive,\Omega)$,
\begin{equation}\label{friedrichs}
\|v\|_{W^{1,p}(\Omega)^n}\leq C\left(\|v\|_{L^p(\Omega)^n}+\|\curl v\|_{L^p(\Omega)^n}+\|\dive v\|_{L^p(\Omega)}\right).
\end{equation}
\end{theorem}

\begin{proof} It is sufficient to prove \eqref{friedrichs} for $v\in W^{1,\infty}(\Omega)^n$; the thesis will then follow by density. We recall that the extension $Ev$ is defined a.e. on $\R^n$ since $|\partial\Omega|=0$. Moreover, from the definition of $Ev$ we can suppose that $\supp Ev$ is contained in a ball $B$.\\
Since $Ev\in W^{1,p}(B)^n$, from \eqref{corol1} we have that
\begin{equation*}
\|Ev\|_{L^p(B)^n}+\|\curl(Ev)\|_{L^p(B)^n}+\|\dive(Ev)\|_{L^p(B)}\leq C\left(\|v\|_{L^p(\Omega)^n}+\|\curl v\|_{L^p(\Omega)^n}+\|\dive v\|_{L^p(\Omega)}\right).
\end{equation*}
Hence, from Friedrichs inequality for smooth domains and the above inequality, we get
\begin{align*}
\|v\|_{W^{1,p}(\Omega)^n}&=\|Ev\|_{W^{1,p}(\Omega)^n}\leq\|Ev\|_{W^{1,p}(B)^n}\leq C(\|Ev\|_{L^p(B)^n}+\|\curl(Ev)\|_{L^p(B)^n}+\|\dive(Ev)\|_{L^p(B)})\\[2mm]
&\leq C\left(\|v\|_{L^p(\Omega)^n}+\|\curl v\|_{L^p(\Omega)^n}+\|\dive v\|_{L^p(\Omega)}\right),
\end{align*}
i.e. the thesis.
\end{proof}

We conclude this section by proving Gaffney inequality as a direct consequence of Theorem \ref{fridis}.

\begin{theorem}[Gaffney inequality]\label{gaffin} Let $\Omega\subset\R^n$ be a bounded simply connected $(\varepsilon,\delta)$ domain with $\partial\Omega$ a $d$-set. Let $v\in W^{1,p}(\Omega)^n$ be such that $v\in W^p(\curl,\Omega)\cap W^p_0(\dive,\Omega)$. Then there exists $C=C(\varepsilon,\delta,n,p,\Omega)>0$ such that
\begin{equation}\label{gaffney}
\|v\|_{W^{1,p}(\Omega)^n}\leq C\left(\|\curl v\|_{L^p(\Omega)^n}+\|\dive v\|_{L^p(\Omega)}\right).
\end{equation}
\end{theorem}

\begin{proof} We argue by contradiction. Let us suppose that \eqref{gaffney} does not hold; hence, there exists a sequence of vectors $\{v_k\}\subset W^{1,p}(\Omega)^n\cap W^p(\curl,\Omega)\cap W^p_0(\dive,\Omega)$ such that
\begin{equation*}
\|v_k\|_{W^{1,p}(\Omega)^n}=1\quad\text{and}\quad\|\curl v_k\|_{L^p(\Omega)^n}+\|\dive v_k\|_{L^p(\Omega)}\xrightarrow[k\to+\infty]{} 0.
\end{equation*}
Since $\|v_k\|_{W^{1,p}(\Omega)^n}=1$, there exists a subsequence of $\{v_k\}$ (which we still denote by $v_{k}$) such that
\begin{equation*}
v_{k}\rightharpoonup v\,\,\text{in}\,\,W^{1,p}(\Omega)^n\quad\text{and}\quad v_{k}\rightarrow v\,\,\text{in}\,\,L^p(\Omega)^n.
\end{equation*}
Since the distributional limits coincide with the weak limits, it immediately follows that $\dive v=0$ and $\curl v=0$.\\
We now prove that $\{v_k\}$ is a Cauchy sequence in $W^{1,p}(\Omega)^n$. From Friedrichs inequality \eqref{friedrichs}, for every $k,j\in\N$ one has
\begin{equation}\label{stimacauchy}
\|v_k-v_j\|_{W^{1,p}(\Omega)^n}\leq C\left(\|v_k-v_j\|_{L^p(\Omega)^n}+\|\dive (v_k-v_j)\|_{L^p(\Omega)}+\|\curl (v_k-v_j)\|_{L^p(\Omega)^n}\right).
\end{equation}
From the strong convergence of $v_k$ in $L^p(\Omega)^n$, the first term on the right-hand side of \eqref{stimacauchy} vanishes. As to the other two terms, they also vanish since $\curl v_k$ and $\dive v_k$ both tend to 0 in $L^p$ as $k\to+\infty$. Hence $v_k$ is a Cauchy sequence in $W^{1,p}(\Omega)^n$, and $v_k\to v$ strongly in $W^{1,p}(\Omega)^n$.\\
We recall that if $\curl v=0$ in $\Omega$ and $\Omega$ is simply connected, there exists a function $\Phi\in W^{1,p}(\Omega)$ such that $v=\nabla\Phi$. This in turn implies that $\Delta\Phi=\dive\nabla\Phi=\dive v=0$ in $\Omega$. Moreover, since $v\in W^p_0(\dive)$, we also have that
\begin{center}
$\displaystyle\frac{\partial\Phi}{\partial\nu}=\nu\cdot\nabla\Phi=\nu\cdot v=0$ on $\partial\Omega$.
\end{center}
Hence $\Phi\in W^{1,p}(\Omega)$ is the unique weak solution of the following problem
\begin{equation}\label{probphi}
\begin{cases}
\Delta\Phi=0\quad &\text{in}\,\,\Omega,\\[2mm]
\displaystyle\frac{\partial\Phi}{\partial\nu}=0 &\text{on}\,\,\partial\Omega.
\end{cases}
\end{equation}
This implies that $\Phi$ is constant, and so $v=\nabla\Phi=0$ on $\Omega$. We reached a contradiction, since
\begin{equation*}
1=\|v_k\|_{W^{1,p}(\Omega)^n}\xrightarrow[k\to+\infty]{}\|v\|_{W^{1,p}(\Omega)^n}=0.
\end{equation*}

\end{proof}

\subsection{The case $v\in W^p(\dive,\Omega)\cap W^p_0(\curl,\Omega)$}\label{rotore}

We now consider the case $v\in W^p(\dive,\Omega)\cap W^p_0(\curl,\Omega)$. We recall that this implies $\nu\times v=0$ on $\partial\Omega$ in the dual of $B^{p',p'}_\alpha(\partial\Omega)$.\\
We approximate $v\in W^p(\dive,\Omega)\cap W^p_0(\curl,\Omega)\cap W^{1,\infty}(\Omega)^n$ by means of the polynomials $P_j$ as in the previous section. We remark that in this case
\begin{equation*}
\dive P_j=\frac{1}{|Q_j^*|}\int_{Q_j^*} \dive v\,\de x\neq 0.
\end{equation*}

As in the previous subsection, estimates \eqref{int1} and \eqref{poincare} hold, as well as lemmas \ref{lemma3.2}, \ref{lemma1} and \ref{lemma2}, under the hypothesis that $v\in W^p(\dive,\Omega)\cap W^p_0(\curl,\Omega)\cap W^p_0(\curl,S)$ for every $S\in W_1$.\\
For the sake of clarity, we state the analogous of Proposition \ref{proposiz} and Theorem \ref{fridis} in this case.

\begin{prop} For every $v\in W^{1,\infty}(\Omega)^n$ such that $v\in W^p(\dive,\Omega)\cap W^p_0(\curl,\Omega)$ we have
\begin{align}
\|Ev\|_{L^p((\Omega^c)^0)^n}&+\|\dive(Ev)\|_{L^p((\Omega^c)^0)}+\|\curl(Ev)\|_{L^p((\Omega^c)^0)^n}\notag\\[2mm]
&\leq C\left(\|v\|_{L^p(\Omega)^n}+\|\dive v\|_{L^p(\Omega)}+\|\curl v\|_{L^p(\Omega)^n}\right)\label{corol1rot}
\end{align}
and
\begin{equation}\label{corol2rot}
\|Ev\|_{W^{1,\infty}((\Omega^c)^0)^n}\leq C\|v\|_{W^{1,\infty}(\Omega)^n}.
\end{equation}
\end{prop}

\bigskip

\begin{theorem}[Friedrichs inequality]\label{fridisrot} Let $\Omega\subset\R^n$ be a bounded $(\varepsilon,\delta)$ domain with $\partial\Omega$ a $d$-set. There exists a constant $C=C(\varepsilon,\delta,n,p,\Omega)>0$ such that, for every $v\in W^{1,p}(\Omega)^n$ such that $v\in W^p(\dive,\Omega)\cap W^p_0(\curl,\Omega)$,
\begin{equation}\label{friedrichsrot}
\|v\|_{W^{1,p}(\Omega)^n}\leq C\left(\|v\|_{L^p(\Omega)^n}+\|\curl v\|_{L^p(\Omega)^n}+\|\dive v\|_{L^p(\Omega)}\right).
\end{equation}
\end{theorem}

We conclude by proving Gaffney inequality.

\begin{theorem}[Gaffney inequality]\label{gaffinrot} Let $\Omega\subset\R^n$ be a bounded simply connected $(\varepsilon,\delta)$ domain with $\partial\Omega$ a $d$-set. Let $v\in W^{1,p}(\Omega)^n$ be such that $v\in W^p(\dive,\Omega)\cap W^p_0(\curl,\Omega)$. Then there exists $C=C(\varepsilon,\delta,n,p,\Omega)>0$ such that
\begin{equation}\label{gaffneyrot}
\|v\|_{W^{1,p}(\Omega)^n}\leq C\left(\|\curl v\|_{L^p(\Omega)^n}+\|\dive v\|_{L^p(\Omega)}\right).
\end{equation}
\end{theorem}

\begin{proof} We proceed as in the proof of \cite[Corollary 3.51]{monk}; we argue by contradiction and we suppose that \eqref{gaffneyrot} does not hold. As in Theorem \ref{gaffin}, this means that there exists a sequence of vectors $\{v_k\}\subset W^{1,p}(\Omega)^n\cap W^p(\dive,\Omega)\cap W^p_0(\curl,\Omega)$ such that
\begin{equation*}
\|v_k\|_{W^{1,p}(\Omega)^n}=1\quad\text{and}\quad\|\curl v_k\|_{L^p(\Omega)^n}+\|\dive v_k\|_{L^p(\Omega)}\xrightarrow[k\to+\infty]{} 0.
\end{equation*}
This implies that
\begin{equation*}
v_{k}\rightharpoonup v\,\,\text{in}\,\,W^{1,p}(\Omega)^n\quad\text{and}\quad v_{k}\rightarrow v\,\,\text{in}\,\,L^p(\Omega)^n,
\end{equation*}
with $\dive v=0$ and $\curl v=0$.\\
From \eqref{friedrichsrot}, for every $k,j\in\N$ we have that
\begin{equation}\label{stimacauchyrot}
\|v_k-v_j\|_{W^{1,p}(\Omega)^n}\leq C\left(\|v_k-v_j\|_{L^p(\Omega)^n}+\|\dive (v_k-v_j)\|_{L^p(\Omega)}+\|\curl (v_k-v_j)\|_{L^p(\Omega)^n}\right).
\end{equation}
As in the proof of Theorem \ref{gaffin}, all the terms on the right-hand side of \eqref{stimacauchyrot} vanish when $k,j\to+\infty$, hence $\{v_k\}$ is a Cauchy sequence in $W^{1,p}(\Omega)^n$.\\ 
As in the case $v\in W^p(\curl,\Omega)\cap W^p_0(\dive,\Omega)$, there exists a function $\Phi\in W^{1,p}(\Omega)$ such that $v=\nabla\Phi$ and $\Delta\Phi=0$ in $\Omega$. Since in this case $v\in W^p_0(\curl,\Omega)$, we also have that
\begin{center}
$\nu\times\nabla\Phi=\nu\times v=0$ on $\partial\Omega$.
\end{center}
Up to shifting $\Phi$ by a constant, this implies that $\Phi=0$ on $\partial\Omega$ in the trace sense.
Hence $\Phi\in W^{1,p}(\Omega)$ is the unique weak solution of the following problem
\begin{equation}\label{probphirot}
\begin{cases}
\Delta\Phi=0\quad &\text{in}\,\,\Omega,\\[2mm]
\Phi=0 &\text{on}\,\,\partial\Omega.
\end{cases}
\end{equation}
This implies that $\Phi=0$, therefore $v=0$ on $\Omega$ and we reach the contradiction.
\end{proof}

\medskip

\noindent {\bf Acknowledgements.} The authors have been supported by the Gruppo Nazionale per l'Analisi Matematica, la Probabilit\`a e le loro Applicazioni (GNAMPA) of the Istituto Nazionale di Alta Matematica (INdAM).


\begin{thebibliography}{38}

\bibitem{amrouche} C. Amrouche, C. Bernardi, M. Dauge, V. Girault, \emph{Vector potentials in three-dimensional non-smooth domains}, Math. Methods Appl. Sci., 21 (1998), 823--864.

\bibitem{bauerpauly} S. Bauer, D. Pauly, \emph{On Korn's first inequality for tangential or normal boundary conditions with explicit constants}, Math. Methods Appl. Sci., 39 (2016), 5695--5704.

\bibitem{CHLTV} S. Creo, M. Hinz, M. R. Lancia, A. Teplyaev, P. Vernole, \emph{Magnetostatic problems in fractal domains}, accepted for publication on Fractals and Dynamics in Mathematics, Science and the Arts published by World Scientific. Available on arXiv: https://arxiv.org/abs/1805.08262

\bibitem{dacorogna} G. Csato, B. Dacorogna, S. Sil, \emph{On the best constant in Gaffney inequality}, J. Funct. Anal., 274 (2018), 461--503.

\bibitem{duran} R. Dur\'an, M. A. Muschietti, \emph{The Korn inequality for Jones domains}, Electron. J. Differential Equations, 127 (2004), 10 pp.

\bibitem{friedrichs} K. O. Friedrichs, \emph{Differential forms on Riemannian manifolds}, Comm. Pure Appl. Math., 8 (1955), 551--590.

\bibitem{gaffney} M. P. Gaffney, \emph{Hilbert space methods in the theory of harmonic integrals}, Trans. Amer. Math. Soc., 78 (1955), 426--444.


\bibitem{Jones} P. W. Jones, {\em Quasiconformal mapping and extendability of functions in Sobolev spaces}, Acta Math., 147 (1981), 71--88.

\bibitem{jonsson91} A. Jonsson, {\em Besov spaces on closed subsets of $\R^n$}, Trans. Amer. Math. Soc., 341 (1994), 355--370.

\bibitem{JoWa} A. Jonsson, H. Wallin, {\em Function Spaces on Subsets of $\mathbb{R}^n$}, Part 1, Math. Reports, vol.2, Harwood Acad. Publ., London, 1984.

\bibitem{JoWa2} A. Jonsson, H. Wallin, {\em The dual of Besov spaces on fractals}, Studia Math., 112 (1995), 285--300.

\bibitem{LaVe2} M. R. Lancia, P. Vernole, {\em Semilinear fractal problems: approximation and regularity results}, Nonlinear Anal., 80 (2013), 216--232.

\bibitem{LVstokes} M. R. Lancia, P. Vernole, {\em The Stokes problems in fractal domains: asymptotic behavior of the solutions}, accepted for publication on Discrete Contin. Dyn. Syst. Ser. S.

\bibitem{monk} P. Monk, \emph{Finite Element Methods for Maxwell's Equations}, Oxford University Press, New York, 2003.

\bibitem{NPW15} P. Neff, D. Pauly, K. J. Witsch, \emph{Poincar\'e meets Korn via Maxwell: Extending Korn's first inequality to incompatible tensor fields}, J. Diff. Eq., 258 (2015), 1267--1302.

\bibitem{Schw16}
B. Schweizer, \emph{On Friedrichs inequality, Helmholtz decomposition,
vector potentials, and the div-curl lemma}, preprint (2016), TU Dortmund.

\bibitem{stein} E. M. Stein, \emph{Singular Integrals and Differentiability Properties of Functions}, Princeton University Press, Princeton, New Jersey, 1970.

\bibitem{temam} R. Temam, {\em Navier-Stokes Equations. Theory and Numerical Analysis}, Studies in Mathematics and its Applications, 2, North-Holland Publishing Co., Amsterdam-New York, 1979.

\bibitem{whitney} H. Whitney, \emph{Analytic extensions of differentiable functions defined in closed sets}, Trans. Amer. Math. Soc., 36 (1934), 63--89.





\end{thebibliography}
\end{document}